\documentclass[12pt]{article}
\usepackage{amssymb,amsmath,latexsym}
\usepackage{amsthm}
\usepackage{fontenc}
\numberwithin{equation}{section}
\newtheorem{theorem}{Theorem}

\begin{document}
\author{Alexander E Patkowski}
\title{On Some Families of Integrals Connected to the Hurwitz Zeta Function}
\date{\vspace{-5ex}}
\maketitle
\abstract{Expressions for a family of integrals involving the Hurwitz zeta function are established using standard properties of the Fourier transform.}

\section{Introduction}
\par The Hurwitz zeta function is defined by 
\begin{equation} \zeta(s,a)=\sum_{n\ge0}\frac{1}{(a+n)^s}\end{equation}
for $s\in\mathbb{C},$ $\Re(s)>1,$ and $a$ is chosen appropriately so there are no singularities in the series. $\zeta(s,a)$ admits the integral representation
\begin{equation} \zeta(s,a)=\frac{1}{\Gamma(s)}\int_{0}^{\infty}\frac{e^{-at}}{1-e^{-t}}t^{s-1}dt,\end{equation}
where $\Gamma(s)$ is Euler's gamma function, which is valid for $\Re(s)>1$ and $\Re(a)>0.$ Hermite proved
an interesting integral representation, which actually provides an explicit realization of the analytic continuation to $\mathbb{C}-\{1\}$ and $\Re(a)>0:$
\begin{equation}\zeta(s,a)=\frac{a^{-s}}{2}+\frac{a^{1-s}}{s-1}+2\int_{0}^{\infty}\frac{\sin(s\tan^{-1}(t/a))dt}{(a^2+t^2)^{s/2}(e^{2\pi t}-1)}.\end{equation}
\par The function $\zeta(s,a)$ is analytic for $s\neq1,$ and direct differentiation of (1.3) yields
\begin{equation}\zeta'(s,a)=-\frac{a^{-s}\ln a}{2}-\frac{a^{1-s}\ln a}{s-1}-\frac{a^{1-s}}{(s-1)^2}-2a^{1-s}\ln a\int_{0}^{\infty}\frac{\sin(s\tan^{-1}(t))dt}{(1+t^2)^{s/2}(e^{2a\pi t}-1)}\end{equation}
$$+2a^{1-s}\int_{0}^{\infty}\frac{\cos(s\tan^{-1}(t))\tan^{-1}(t)dt}{(1+t^2)^{s/2}(e^{2a\pi t}-1)}-a^{1-s}\int_{0}^{\infty}\frac{\sin(s\tan^{-1}(t))\ln(t^2+1)dt}{(1+t^2)^{s/2}(e^{2a\pi t}-1)},$$
where $\zeta'(s,a)$ denotes $\partial \zeta(s,a)/\partial s.$
\par The work presently discussed is a continuation of [2, 4, 7] where these integral representations have been employed to evaluate interesting definite integrals. General information about $\zeta(s,a)$ can be found in [1], [5] and [6].
\par The main result is presented next.
\begin{theorem} Let $n\in\mathbb{N}_0.$ For $\Re(a)>0$ and $0\le 2n <\Re(s),$ define
\begin{equation} S_n(a,s):=\int_{0}^{\infty}\frac{t^{2n}\sin(s\tan^{-1}(t/a))dt}{(a^2+t^2)^{s/2}(e^{2\pi t}-1)}.\end{equation}
Then
\begin{equation} S_n(a,s)=\frac{1}{2}\sum_{m=0}^{2n}(-1)^{m+n}\binom{2n}{m}a^{m}P_1(a,m+s-2n),\end{equation}
where
\begin{equation}P_1(a,s)=\zeta(s,a)-\frac{a^{-s}}{2}-\frac{a^{1-s}}{s-1}.\end{equation}
\end{theorem}
Observe that (1.3) corresponds to the special case $n=0$ in (1.5). Here we note that $S_n(a,s)$ is analytic in the set $\{ n\in\mathbb{N}_0, 0\le 2n <\Re(s): s-2n\neq 1\}.$

\par The proof of Theorem 1.1 is based on identifying the Fourier sine transform of two special functions and then apply the corresponding Parseval identity. Recall that for a function defined on the half-line, the Fourier sine transform is 
\begin{equation} \mathfrak{F}(f)(w):=\sqrt{\frac{2}{\pi}}\int_{0}^{\infty}f(t)\sin(wt)dt,\end{equation}
provided the integral converges. The corresponding Parseval identity states that 

\begin{equation}\int_{0}^{\infty}\mathfrak{F}(f)(w)\mathfrak{F}(g)(w)dw=\int_{0}^{\infty}f(t)g(t)dt.\end{equation}

\par Theorem 1.1 is a direct consequence of Parseval's relation applied to the functions
$$f(t)=1/(e^{2\pi t}-1), ~~~~~~~~\mbox{and}~~~~~~~~ g(t)=\frac{t^{2n}\sin(s\tan^{-1}(t/a))}{(a^2+t^2)^{s/2}}.$$

The Fourier sine transform $f$ comes from entry 3.951.12 of [3]. It states an equivalent form of the identity
\begin{equation} \int_{0}^{\infty}\frac{\sin(wt)dt}{e^{2\pi t}-1}=\frac{1}{2}\left(\frac{1}{e^{w}-1}+\frac{1}{2}-\frac{1}{w}\right).\end{equation}
The Fourier sine transform of $g(t)$ is given in terms of the \it associated Laguerre polynomials \rm $L_n^{k}(x)$ defined by the Rodrigues representation
\begin{equation} L_n^{k}(x)=\frac{e^{x}x^{-k}}{n!} \frac{d^{n}}{dx^n}\left(e^{-x}x^{n+k}\right),\end{equation}
for $n\in\mathbb{N}\cup \{0\}.$
\par Theorem 1.1 is extended in Section 3 to include integrals in which the kernel $1/(e^{2\pi t}-1)$ is replaced by 
$$1/(e^{\pi t}+1), ~~~~~~~~1/\sinh(\pi t), ~~~~~~~~1/\cosh(\pi t).$$
Consider the families of integrals
\begin{equation} I_k(q)=\int_{0}^{\infty}\frac{tdt}{(1+t^2)^{k+1}(e^{2\pi q t}-1)},\end{equation}
\begin{equation} T_k(q)=\int_{0}^{\infty}\frac{t^{k}\tan^{-1}tdt}{(e^{2\pi q t}-1)},\end{equation}
\begin{equation} L_k(q)=\int_{0}^{\infty}\frac{t^{k}\ln (1+t^2) dt}{(e^{2\pi q t}-1)}.\end{equation}
The reader will find in [2] explicit expression for $I_k(q)$ in terms of the derivatives of the polygamma function and for $T_{2k}(q)$ and $L_{2k+1}(q)$ 
remains an open problem. It would be of interest to analyze the evaluations discussed here in relation to this open problem.

\section{The Proof}
 The proof of Theorem 1.1 is based on the computation of two Fourier sine transforms. Formula 3.951.12 in [3] states an equivalent form of the identity 
 \begin{equation} \int_{0}^{\infty}\frac{\sin(wt)dt}{e^{2\pi t}-1}=\frac{1}{2}\left(\frac{1}{e^{w}-1}+\frac{1}{2}-\frac{1}{w}\right).\end{equation}
 which gives the sine transform of 
 \begin{equation} f(t)=\frac{1}{e^{2\pi t}-1}, \end{equation}
as 
\begin{equation} \mathfrak{F}(f)(w)=\frac{1}{\sqrt{2\pi}}\left(\frac{1}{e^{w}-1}+\frac{1}{2}-\frac{1}{w}\right).\end{equation}
\par The second Fourier sine transform is that of the associated Laguerre polynomials (1.12). The explicitly formula
\begin{equation} L_n^{k}(x)=\sum_{j=0}^{n}\frac{(-1)^j(n+k)!}{(n-j)!(k+j)!j!}x^j\end{equation}
is employed in the derivation.
\par Formula 3.769.4 of [3] contains the integral representation
\begin{equation} \int_{0}^{\infty}t^{2n}\left((a-it)^{-s}-(a+it)^{-s}\right)\sin(wt)dt=\frac{(-1)^ni\pi(2n)!}{\Gamma(s)e^{aw}w^{2n+1-s}}L_{2n}^{s-2n-1}(aw),\end{equation}
for $w>0$ $\Re(a)>0$ and $0\le 2n<\Re(s).$ The integrand can be simplified using  
 $$(a-it)^{-s}-(a+it)^{-s}=\frac{2i\sin(s\tan^{-1}(t/a))}{(a^2+t^2)^{s/2}}.$$
Therefore (2.5) can be written as
\begin{equation} \int_{0}^{\infty}t^{2n}\frac{\sin(s\tan^{-1}(t/a))}{(a^2+t^2)^{s/2}}\sin(wt)dt=\frac{(-1)^n\pi(2n)!}{2\Gamma(s)e^{aw}w^{2n+1-s}}L_{2n}^{s-2n-1}(aw).\end{equation}
Or equivalently we may state that the Fourier sine transform of
\begin{equation} g(t)=\frac{t^{2n}\sin(s\tan^{-1}(t/a))}{(a^2+t^2)^{s/2}},\end{equation}
is given by
\begin{equation} \mathfrak{F}(g)(w)=\frac{(-1)^n\pi(2n)!}{2\Gamma(s)e^{aw}w^{2n+1-s}}L_{2n}^{s-2n-1}(aw).\end{equation}
\par Parseval's identity (1.9) gives the next result. \newline
{\bf Lemma 2.1.} \it For $\Re(a), \Re(s)>0,$
\begin{equation}\int_{0}^{\infty}\frac{t^{2n}\sin(s\tan^{-1}(t/a))dt}{(a^2+t^2)^{s/2}(e^{2\pi t}-1)}=\frac{(-1)^n(2n)!}{2\Gamma(s)}\int_{0}^{\infty}e^{-aw}w^{-2n-1+s}L_{2n}^{s-2n-1}(aw)\left(\frac{1}{e^{w}-1}+\frac{1}{2}-\frac{1}{w}\right)dw.\end{equation}
\rm
The explicit formula (2.4) for the Laguerre polynomials is now employed to evaluate the integral on the right side of Lemma 2.1.

$$\int_{0}^{\infty} \frac{e^{-aw}w^{-2n-1+s}L_{2n}^{s-2n-1}(aw)}{e^{w}-1}dw$$
$$=\sum_{j=0}^{2n}\frac{(-1)^j(s-1)!a^j}{(2n-j)!(s-2n-1+j)!j!}\int_{0}^{\infty}\frac{w^{s-2n-1+j}e^{-(a+1)w}dw}{1-e^{-w}}$$
$$=\sum_{j=0}^{2n}\frac{(-1)^j(s-1)!a^j}{(2n-j)!j!}\zeta(s-2n+j,a+1).$$
In the last step we have employed the integral representation for the Hurwitz zeta function (1.2). For the desired formula we must write $1/(e^{w}-1)=e^{w}/(e^{w}-1)-1.$ The remaining integrals corresponding to the
terms $1/2$ and $1/w$ are elementary, and so are omitted. 

\section{Related Integrals}
In this section we produce results similar to Theorem 1.1 for a family of integrals of the form
$$\int_{0}^{\infty}f(t)K(t)dt,$$
where the kernel $1/(e^{2\pi t}-1)$ in Theorem 1.1 is replaced by 
$$1/(e^{\pi t}+1), ~~~~~~~~1/\sinh(\pi t), ~~~~ \rm or ~~~~1/\cosh(\pi t).$$

The next lemma will be needed for future computations corresponding to these kernels.
{\bf Lemma 3.1} \it Assume $\Re(s)>1$ and $\Re(a)\ge0.$ Then
\begin{equation}\int_{0}^{\infty}\frac{t^{s-1}e^{-at}}{\sinh(t)}dt=\Gamma(s)\left(\zeta(s,a)-2^{-s}\zeta(s,a/2)\right).\end{equation}
If $\Re(a)>0,$ then
\begin{equation}\int_{0}^{\infty}\frac{t^{s-1}e^{-at}}{1+e^{-t}}dt=\Gamma(s)\left(-\zeta(s,a)+2^{1-s}\zeta(s,a/2)\right),\end{equation}
and
\begin{equation}\int_{0}^{\infty}\frac{t^{s-1}e^{-at}}{\cosh(t)}dt=\Gamma(s)2^{-2s}\left(\zeta(s,\frac{1+a}{4})-\zeta(s,\frac{a+3}{4})\right).\end{equation}
\rm
These integrals are well-known variations of (1.2). Details are in [2].

\begin{theorem} Let $n\in\mathbb{N}_0.$ For $\Re(a)>0$ and $0\le 2n <\Re(s),$ define
\begin{equation} SH_n(a,s):=\int_{0}^{\infty}\frac{t^{2n}\sin(s\tan^{-1}(t/a))dt}{(a^2+t^2)^{s/2}\sinh(\pi t)}.\end{equation}
Then
\begin{equation} SH_n(a,s)=\frac{1}{2}\sum_{m=0}^{2n}(-1)^{m+n}\binom{2n}{m}a^{m}P_2(a,m+s-2n),\end{equation}
where
\begin{equation}P_1(a,s)=2^{2-s}\zeta(s,a/2)-2\zeta(s,a)-a^{-s}.\end{equation}
\end{theorem}
\begin{proof} The identity
\begin{equation} \int_{0}^{\infty} \frac{\sin(wt)}{\sinh(\beta t)}=\frac{\pi}{2\beta}\tanh{\frac{\pi w}{2\beta}}\end{equation}
appears as entry 3.981.1 in [3]. 
\par The value $\beta=\pi$ in (3.7) shows that the sine Fourier transform of $1/\sinh(\pi t)$ is $\frac{1}{2}\tanh(w/2).$ Then
(1.9) and (2.5) give
\begin{equation}\int_{0}^{\infty}\frac{t^{2n}\sin(s\tan^{-1}(t/a))dt}{(a^2+t^2)^{s/2}\sinh(\pi t)}=\frac{(-1)^n\pi(2n)!}{4\Gamma(s)}\int_{0}^{\infty}\tanh(\frac{w}{2})e^{-aw}w^{-2n-1+s}L_{2n}^{s-2n-1}(aw)\left(\frac{1}{e^{w}-1}+\frac{1}{2}-\frac{1}{w}\right)dw.\end{equation}
Now use
$$\tanh(w/2)=\frac{2}{1+e^{-w}}-1$$ \end{proof}
and proceed as in the proof of Theorem 1.1.
\par The next results are established along similar lines of the proof presented above. The details are omitted. Entries 3.911.1 
and 3.981.3 in [3] are 
\begin{equation} \int_{0}^{\infty} \frac{\sin(wt)}{e^{\beta t}+1}=\frac{1}{2w}-\frac{\pi}{2\beta\sinh{\frac{\pi w}{\beta}}},\end{equation}
and 
\begin{equation} \int_{0}^{\infty} \frac{\cos(wt)}{\cosh(\beta t)}=\frac{\pi}{2\beta\cosh{\frac{\pi w}{2\beta}}},\end{equation}
respectively. These are used instead of (3.7) in the proofs.
\begin{theorem} Let $n\in\mathbb{N}_0.$ For $\Re(a)>0$ and $0\le 2n <\Re(s),$ define
\begin{equation} EP_n(a,s):=\int_{0}^{\infty}\frac{t^{2n}\sin(s\tan^{-1}(t/a))dt}{(a^2+t^2)^{s/2}(e^{2\pi t}+1)}.\end{equation}
Then
\begin{equation} EP_n(a,s)=\frac{1}{2}\sum_{m=0}^{2n}(-1)^{m+n}\binom{2n}{m}a^{m}P_1(a,m+s-2n),\end{equation}
where
\begin{equation}P_3(a,s)=\frac{a^{1-s}}{s-1}-\zeta(s,a)-2^{-s}\zeta(s,a/2).\end{equation}
\end{theorem}

\begin{theorem} Let $n\in\mathbb{N}_0.$ For $\Re(a)>0$ and $0\le 2n <\Re(s),$ define
\begin{equation} CH_n(a,s):=\int_{0}^{\infty}\frac{t^{2n}\sin(s\tan^{-1}(t/a))dt}{(a^2+t^2)^{s/2}\cosh(\pi t/2)}.\end{equation}
Then
\begin{equation} CH_n(a,s)=\frac{1}{2}\sum_{m=0}^{2n}(-1)^{m+n}\binom{2n}{m}a^{m}P_1(a,m+s-2n),\end{equation}
where
\begin{equation}P_4(a,s)=\frac{1}{2^{2s}}\left(\zeta(s,\frac{a+1}{4})-\zeta(s,\frac{a+3}{4})\right).\end{equation}
\end{theorem}

The final result describes integrals containing odd powers of $t$ in the integrand. As before, the proofs are similar to that of 
Theorem 1.1, so they are omitted.
\begin{theorem} Let $n\in\mathbb{N}_0.$ For $\Re(a)>0$ and $-1\le 2n+1 <\Re(s),$ then
\begin{equation} \int_{0}^{\infty}\frac{t^{2n+1}\cos(s\tan^{-1}(t/a))dt}{(a^2+t^2)^{s/2}(e^{2\pi t}-1)}=\frac{1}{2}\sum_{m=0}^{2n+1}(-1)^{m+n}\binom{2n+1}{m}a^{m}P_1(a,m+s-2n-1),\end{equation}
\begin{equation} \int_{0}^{\infty}\frac{t^{2n+1}\cos(s\tan^{-1}(t/a))dt}{(a^2+t^2)^{s/2}\sinh(\pi t)}=\frac{1}{2}\sum_{m=0}^{2n+1}(-1)^{m+n}\binom{2n+1}{m}a^{m}P_2(a,m+s-2n-1),\end{equation}
\begin{equation} \int_{0}^{\infty}\frac{t^{2n+1}\cos(s\tan^{-1}(t/a))dt}{(a^2+t^2)^{s/2}(e^{\pi t}+1)}=\frac{1}{2}\sum_{m=0}^{2n+1}(-1)^{m+n}\binom{2n+1}{m}a^{m}P_3(a,m+s-2n-1),\end{equation}
and if $\Re(a)>0$ and $0\le 2n <\Re(s)-1,$
\begin{equation} \int_{0}^{\infty}\frac{t^{2n+1}\sin(s\tan^{-1}(t/a))dt}{(a^2+t^2)^{s/2}\cosh(\pi t/2)}=\frac{1}{2}\sum_{m=0}^{2n+1}(-1)^{m+n}\binom{2n+1}{m}a^{m}P_4(a,m+s-2n-1).\end{equation}
\end{theorem}

1390 Bumps River Rd. \\*
Centerville, MA
02632 \\*
USA \\*
E-mail: alexpatk@hotmail.com

\begin{thebibliography}{9}
\bibitem{ConcreteMath} G. Andrews, R. Askey, and R. Roy, \emph{Special Functions, volume 71 of Encyclopedia of Mathematics and its Applications.} Cambridge University Press, New York, 1999.
\bibitem{ConcreteMath} G. Boros, O. Espinosa, and V. Moll, \emph{On some families of integrals solvable in terms of polygamma and negapolygamma functions.} Integrals Transforms and Special Functions, 14:187--203, 2003.
\bibitem{ConcreteMath}I. S. Gradshteyn and I. M. Ryzhik. \emph{Table of Integrals, Series, and Products}. Edited by A. Jeffrey and D. Zwillinger. Academic Press, New York, 7th edition, 2007.
\bibitem{ConcreteMath}J. Zhang S. Kanemitsu, Y. Tanigawa. \emph{Evaluation of Spanen integrals of the product of zeta functions,} Integrals Transforms and Special Functions,-- 19:115--128, 2008.
\bibitem{ConcreteMath} H. M. Srivastava and J. Choi. \emph{Series associated with the zeta and related functions.} Kluwer Academic Publishers, 1st edition, 2001.
\bibitem{ConcreteMath} E.T. Whittaker and G.N. Watson. \emph{Modern Analysis}. Cambridge University Press, 1962.
\bibitem{ConcreteMath} N. Y. Zhang and K. S. Williams. \emph{Special values of the Lerch zeta function and the evaluation
of certain integrals.} Proc. Amer. Math. Soc., 119:35--49, 1993.

\end{thebibliography}
\end{document}